\documentclass{article}

\usepackage{arxiv}

\usepackage[utf8]{inputenc} 
\usepackage[T1]{fontenc}    
\usepackage{hyperref}       
\usepackage{url}            
\usepackage{booktabs}       
\usepackage{amsfonts}       
\usepackage{nicefrac}       
\usepackage{microtype}      
\usepackage{graphicx}
\usepackage{natbib}
\usepackage{doi}

\usepackage{amssymb}

\newtheorem{theorem}{Theorem}

\newtheorem{corollary}{Corollary}[theorem]

\newcommand{\qed}{\hfill$\Box$}
\newenvironment{proof}[1][Proof]{\textit{\textbf{#1:}} }

\title{Exploring structural properties of $k$-trees and block graphs }

\author{Lilian Markenzon \\
	NCE,PPGI - Universidade Federal do Rio de Janeiro\\ 
Rio de Janeiro, RJ, Brazil\\
	\texttt{markenzon@nce.ufrj.br} \\
	\And
	Allana S. S. Oliveira \\
Instituto Federal Sudeste De Minas \\ 
Campus Juiz De Fora, Juiz De Fora, MG, Brazil.\\
	\texttt{allana.sthel@gmail.com} \\
	 \AND
	Cybele T.M.Vinagre \\
	 Instituto de Matem\'atica e Estat\'{\i}stica \\
	 Universidade Federal Fluminense, Niter\'oi, RJ, Brazil \\
	\texttt{cybele\_vinagre@id.uff.br} \\
}

\renewcommand{\headeright}

\hypersetup{
pdftitle={A template for the arxiv style},
pdfsubject={q-bio.NC, q-bio.QM},
pdfauthor={David S.~Hippocampus, Elias D.~Striatum},
pdfkeywords={First keyword, Second keyword, More},
}

\begin{document}
\maketitle

\begin{abstract}
	We present a new characterization of $k$-trees based on their  reduced clique graphs and 
$(k+1)$-line graphs, which are block graphs.
We explore structural properties of these two classes,  
showing that the number of clique-trees of a $k$-tree $G$ equals the number of spanning trees
of the $(k+1)$-line graph of $G$.  
This relationship allows to present  a new approach for determining
the number of spanning trees of any connected block graph.
We show that these results can be accomplished in linear time complexity.
\end{abstract}

$k$-tree, block graph, reduced clique graph, $k$-line graph, number of clique trees, number of spanning trees

\section{Introduction}
The class of $k$-trees, introduced in 1968 by Beineke and Pippert \cite{beineke1969number},
has an inductive definition which naturally extends the definition of a tree.
It is a subclass of chordal graphs and it has been the subject of considerable research. 
In this paper we present a new characterization of $k$-trees, based on two different associated structures: 
the  reduced clique graph, notion introduced by Galinier \textit{et al.} in \cite{galinier1995chordal}, 
and the $k$-line graph, the generalization of the line graph operation introduced by L\^{e} in \cite{le1993perfect}.

A clique tree representation of a chordal graph 
is particularly useful, allowing the development of efficient algorithms that take advantage 
of the compactness of the representation.
In \cite{ho1989counting}, an algorithm that generates any clique tree of a chordal graph $G$ is presented, 
 based upon the assumption that the sets of maximal cliques and minimal vertex separators of $G$ are given. 
From this result, the authors derive an exact formula 
of counting  the number of clique trees of a labeled connected chordal graph. 
In this paper, we resume the formula, showing that it can be simplified for $k$-trees.

Further, the number of spanning trees of a connected graph is a well-known invariant,
provided by the Kirchhoff's matrix tree theorem \cite{biggs1993algebraic}, that can be computed in polynomial time as any cofactor of the Laplacian matrix of the graph. 
In this paper, we show that if $G$ is a $k$-tree, then the number of clique trees of $G$ equals 
the number of spanning trees of the $(k+1)$-line graph of $G$, which is a block graph as proved in \cite{oliveira2021k+}. Moreover, we calculate the number of spanning trees of any connected block graph. Both results are computed in linear time complexity.

\section{Basic concepts}
\label{sec:headings}

Let $G  =(V,E)$,  be a  graph, 
where $|V|=n$  and  $|E| = m$. 
The {\em set of neighbors\/} of a vertex $v \in V$ is denoted by
$N(v) = \{ w \in V; \{v,w\} \in E\}$. 
The \emph{degree} of a vertex $v\in V$ is $d(v)=|N(v) |$. 
For any $S \subseteq V$, 
the subgraph of $G$ induced by $S$ is denoted $G[S]$. 
 If $G[S]$ is a complete graph then $S$ is a \emph{clique} in $G$. 
A graph is said to be $H$-free if it contains no $H$ as an induced subgraph.
A vertex $v\in V$ is said to be {\em
simplicial\/} in $G$ when $N(v)$ is a clique in $G$.

Basic concepts about chordal graphs are assumed to be known and 
can be found in  Blair and Peyton \cite{blair1993introduction}  and Golumbic \cite{golumbic2004algorithmic}.
In this section, the most pertinent concepts are reviewed.

A subset $S \subset V$ is
a {\em separator} of $G$ if at least two vertices in the same connected
component of $G$ are in two distinct connected components of
$G[V\setminus S]$. 

Let $G = (V, E)$ be a chordal graph and $u,v  \in V$. 
A subset $S \subset V$  is a {\em vertex separator}  for
non-adjacent vertices $u$  and $v$  (a $uv$-separator) if the
removal of $S$ from the graph separates $u$ and $v$  into distinct
connected components. 
If no proper subset of $S$  is a $uv$-separator then $S$ is a {\em minimal $uv$-separator}. 
When the pair of vertices remains unspecified, we refer to $S$  as a {\em
minimal vertex separator} ({\em mvs}). 
The set of minimal vertex separators is denoted by $\mathbb{S}$.

The {\it clique-intersection graph\/} of a chordal graph $G$ is the
connected weighted graph whose vertices are the maximal cliques of $G$ and whose
edges connect vertices corresponding to non-disjoint maximal cliques.
Each edge is assigned an integer weight, given by the cardinality of the
intersection between the maximal cliques represented by its endpoints.
Every maximum-weight spanning tree of the clique-intersection graph of $G$
is called a {\it clique tree\/} of $G$.
The set of maximal cliques of $G$ is denoted by $\mathbb{Q}$.
A {\em simplicial clique} is a maximal clique containing at least one simplicial vertex.

For a chordal graph $G$ 
and a clique tree $T$ of $G$,
a set $S\subset V$ is a \emph{mvs}  of $G$ if
and only if $S= Q\cap Q' $ for some edge $\{Q, Q'\}$ in $T$. 
Moreover, the multiset  ${\mathbb M}$ of
the minimal vertex separators of $G$ is the same for every
clique tree of $G$.
The {\em multiplicity} of the minimal vertex separator $S$, denoted by $ \mu (S)$,   
is the number of times that $S$ appears in  ${\mathbb M}$. 
The determination of the minimal vertex separators and their multiplicities 
can be performed in linear time \cite{markenzon2010one}.

The maximal cliques $Q$ and $Q^{\prime}$ of $G$ form  a {\em separating  pair}
 if $Q \cap Q^{\prime} \neq \emptyset$ and every path in $G$ from a vertex of
 $Q\backslash  Q^{\prime}$ to a vertex of $Q^{\prime} \backslash Q$ contains a vertex of $Q \cap Q^{\prime}$.
The {\em reduced clique graph} $C_{r}(G)$ 
 of $G$, introduced by Galinier \textit{et al.} in \cite{galinier1995chordal}, is the graph whose vertices are maximal cliques of $G$
and whose edges $\{Q,Q^{\prime}\}$ are between cliques $Q$ and $Q^{\prime}$ forming separating pairs.

\begin{theorem}{\bf\cite{habib2012reduced}} \label{theo:habib1}
A set $S$ is a minimal vertex separator of a chordal graph $G$ if and only if there exist maximal cliques 
$Q$ and $Q^{\prime}$ of $G$ forming a separating pair such that $S=Q \cap Q^{\prime}$.
\end{theorem}

\begin{theorem} {\bf\cite{galinier1995chordal}} \label{theo:cliqueTree_spanninTree}
Let $G$ be a connected chordal graph. 
A tree $T$ is a clique tree of $G$ if and only if $T$ is a maximum-weight spanning tree of $C_r(G)$ 
where the weight of each edge $\{Q,Q'\}$ is defined as $|Q\cap Q'|$. 
Moreover, the reduced clique graph $C_r(G)$ is precisely the union of all clique trees of $G$.
\end{theorem}

As Theorem \ref{theo:cliqueTree_spanninTree} states, not all spanning trees of $C_r(G)$ are clique trees of $G$. 
In Figure \ref{fig:exem_CrG_e_spantree}, for example, the tree $T$ is a spanning tree of $C_r(G)$ but $T$ is not a clique tree of $G$.

\begin{figure}[h]
\centering
\includegraphics[scale=0.23]{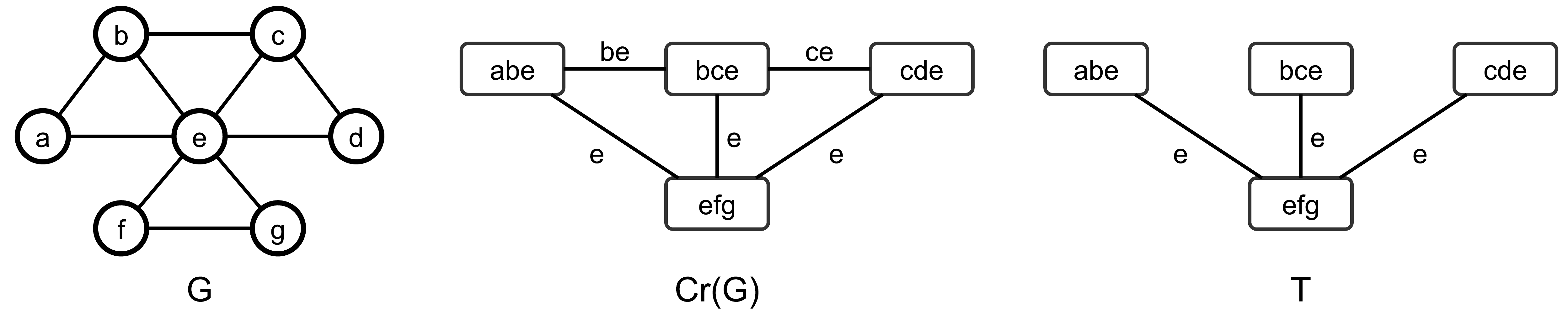}
\caption{$G$, $C_r(G)$ and a spanning tree $T$ of $C_r(G)$.}
\label{fig:exem_CrG_e_spantree}
\end{figure}

The generalization of the line graph operation which we  apply in our work is the one introduced by L\^{e} in \cite{le1993perfect}: for an integer $k\geq2$, the \emph{$k$-line graph} of $G$, denoted by $\ell_k(G)$, is the graph whose vertices are the  $k$-cliques in $G$ and where two distinct such vertices are adjacent if and only if they have, in $G$, $k-1$ vertices in common.

\section{ \emph{k}-trees and their  \emph{(k+1)}-line graphs and reduced clique graphs}

A \emph{$k$-tree} is a chordal graph that can be recursively defined as follows. 
The complete graph with $k+1$ vertices is a $k$-tree. 
A $k$-tree with $n+1$ vertices $(n\geq k+1)$ can be constructed from a $k$-tree with $n$ vertices by adding 
a vertex adjacent to all vertices of a $k$-clique $C$ of the existing $k$-tree, and only to these vertices. 
Theorems \ref{theo:Rose} and \ref{theo:cardinalidadeSeparadoresKarvores} are characterizations of $k$-trees.

\begin{theorem}{\rm\cite{rose1974simple}}\label{theo:Rose}
A chordal graph $G = (V, E)$ is a $k$-tree if and only if
\begin{enumerate}
\item $G$ is connected,
\item $G$ has a $k$-clique but no $k + 2$ clique,
\item every minimal vertex separator of $G$ is a $k$-clique.
\end{enumerate}
\end{theorem}

\begin{theorem}  \label{theo:cardinalidadeSeparadoresKarvores}
Let $G$ be a chordal graph, $\mathbb{Q}$ the set of its maximal cliques and $\mathbb{S}$ the set of its minimal vertex separators.
Graph G is a $k$-tree if and only if
$|Q|=k+1$, for all $Q \in \mathbb{Q}$,  and $|S|=k$, for all $S \in \mathbb{S}$.
\end{theorem}

In \cite{harary1969graph}, it is proved that a graph is the line graph of a tree if and only if it is
 a  connected block graph which is $K_{1,3}$-free. In \cite{oliveira2021k+}, the graphs that are  $(k+1)$-line graphs of  general $k$-trees are characterized. 
Next theorem merges both results.

\begin{theorem}{\rm\cite{oliveira2021k+},\cite{harary1969graph}}\label{teo:grafo(k+1)linhaDeKarvore}
Let $k\geq 1$ be an integer. 
A graph is the $(k+1)$-line graph of a $k$-tree if and only if it is a
connected block graph which is $K_{1,k+2}$-free.
\end{theorem}

In the next theorem we characterize $k$-trees in terms of its reduced clique graph and its $(k+1)$-line graph.

\begin{theorem}\label{theo:carac_cr_klinha}
Let $G$ be a connected chordal graph.
Graph $G$ is a $k$-tree if and only if $C_r(G)$ and $\ell_{k+1}(G)$ are the same.
\end{theorem}

\begin{proof}
By Theorem \ref{theo:Rose}, all  maximal cliques of $G$ have cardinality $k+1$;
by definition, these cliques are the vertices of $C_r(G)$.
There are no cliques of cardinality $(k+2)$ in $G$ (still by Theorem \ref{theo:Rose});
so these maximal  cliques are also the vertices of $\ell_{k+1}(G)$.
The edges of  $C_r(G)$ correspond to minimal vertex separators of $G$ (Theorem \ref{theo:cliqueTree_spanninTree}).
Since $G$ is a $k$-tree, its minimal vertex separators are cliques of cardinality $k$.
All pairs of maximal cliques $Q$ and $Q'$ such that $|Q \cap Q'| = k$ form a separating pair of $G$
and the edge $\{Q,Q'\}$ belongs to $C_r(G)$.
These edges are precisely the edges of $\ell_{k+1}(G)$.


Reciprocally, suppose that $G$ is not a $k$-tree.
Then, by Theorem \ref{theo:Rose}, two cases can occur.

\textbf{Case A:} $G$ has at least one clique $Q$ of cardinality $k+2$.
Suppose that $Q$ is a maximal clique.
In this case, all $k+2$ subcliques of $Q$ with cardinality $k+1$
belong to the set of vertices of  $\ell_{k+1}(G)$ with their intersections of cardinality $k$.
However, as these cliques are not maximal cliques they do not appear in $C_r(G)$.
So $C_r(G)$ is not equal to  $\ell_{k+1}(G)$.
Contradiction.

\textbf{Case B:} there is a $mvs$ in $G$ with cardinality not equal to $k$.
In this case there is an edge belonging to  $C_r(G)$ that does not appear in $\ell_{k+1}(G)$.
Contradiction.\qed
\end{proof}

Notice that, by Theorem \ref{teo:grafo(k+1)linhaDeKarvore}, 
the reduced clique graph of a $k$-tree is a block graph, a fact that is not true for chordal graphs in general as we can see in Figure \ref{fig:exem_CrG_e_spantree}.

\begin{theorem}\label{theo:relation-ktree-block}
Let $G$ be a $k$-tree, $\mathbb S$ its set of minimal vertex separators and  $C_r(G)$ its reduced clique graph.
Then each maximal clique $Q$ of $C_r(G)$ corresponds to a minimal vertex separator of $G$,
and more, the multiplicity of the mvs $S \in \mathbb S$ equals $|Q| -1$.
\end{theorem}
  \begin{proof}
Let $G$ be a $k$-tree. 
Each maximal clique $Q$ of $C_r(G)$ is a block  whose vertices represent
 the maximal cliques of $G$ that share the same \emph{msv} $S$. 
Then, no other edge of $C_r(G)$ corresponds to $S$, since no other maximal clique of $G$ contains $S$.
As each edge of the block $Q$ is associated to $S$, it will appear $|Q|-1$ times in any clique tree.
Hence, $\mu (S)=|Q|-1$.\qed
\end{proof} 

\section{Number of clique trees of a \emph{k}-tree}

In 1989, Ho and Lee \cite{ho1989counting} presented a formula for counting the number of clique trees of a chordal graph.
Kumar and Madhavan \cite{kumar2002clique} modify this formula, focusing in minimal vertex separators and
stating that  the complexity time of the process is $(|{\mathbb S}| (\omega(G)|V| +|E| ))$, 
being $\omega(G)$ the clique number of $G$.
We resume the formula presented in \cite{ho1989counting}, showing that for $k$-trees it can be simplified
and its time complexity reduced.
Previously, some definitions are needed.

Let $G=(V,E)$ be a chordal graph, $\mathbb{Q}$ the set of its maximal cliques and $\mathbb{S}$ the set of its minimal vertex separators. For a set $A \subset V$, $Adj(A) = \cup_{v\in A} N(v) \setminus A$.
For every minimal vertex separator $S \in {\mathbb S}$,\\
${\mathfrak C}_S = \{ C \,|\, C \mbox{ is a connected component} \mbox{of } G\setminus S \mbox{ and } Adj(C) = S\}$\\
and for every $C \in {\mathfrak C}_S$, \\ ${\mathbb Q}_C = \{Q \in {\mathbb Q} \,|\, Q \mbox{ is in } G[C\cup S] 
\mbox{ and }S\subset Q\}.$

\begin{theorem}{\rm\cite{ho1989counting}}
The number of clique trees of a connected chordal graph $G$ is equal to
\begin{equation}\label{eq:formulacliquetreeHoLee}
\prod_{S \in {\mathbb S}} \left[ \left( \sum_{C \in {\mathfrak C}_S} |{\mathbb Q}_C|\right )^ {|{\mathfrak C}_S| -2} 
\cdot  \prod_{C \in {\mathfrak C}_S}  |{\mathbb Q}_C|  \right],
\end{equation}
where $\mathbb S$ is the set of minimal vertex separators of $G$.
\end{theorem}

Actually, the determination of the number of clique trees of $k$-trees depends only on the multiplicity of each $mvs$,
as seen in Theorem  \ref{theo:nro-clique-tree}.

\begin{theorem}\label{theo:nro-clique-tree}
The number of clique trees in a $k$-tree is equal to
\begin{equation}\label{eq:nossaformulacliquetree}
\prod_{S \in {\mathbb S}} \left(\mu(S) +1 \right)^{\mu(S)-1},
\end{equation}
where $\mathbb S$ is the set of minimal vertex separators of $G$.
\end{theorem}

\begin{proof}
Let $G$ be a $k$-tree. 
We are going to apply to $G$ equation (\ref{eq:formulacliquetreeHoLee}). 
It is known that $|S|=k$ for all $S\in {\mathbb S}$ and $|Q|=k+1$ for all 
$Q\in {\mathbb Q}$ (Theorem \ref{theo:cardinalidadeSeparadoresKarvores}). 
Let $S \in {\mathbb S}$ and $C \in {\mathfrak C}_S$.  

Let us observe $C_r(G)$.
The removal of a $mvs$ $S$ corresponds, by Theorem \ref{theo:relation-ktree-block}, 
to the removal of the edges of a maximal clique (a block) of $C_r(G)$.
The vertices that are endpoints of these edges correspond to  maximal cliques of $G$ containing $S$. We can consider two types of remaining vertices in $C_r(G)$:  isolated vertices or vertices that were separators in the original $C_r(G)$.
Either one corresponds to a component in $G[V\setminus S]$ and these vertices are the only ones in each component that contains $S$.
Therefore $|{\mathbb Q}_C|=1$ and the second factor of equation (1) is equal to 1.

The first factor of equation (1) is now 
$\left( \sum_{C \in {\mathfrak C}_S} 1\right )^ {|{\mathfrak C}_S| -2}$.
To consider all $C \in {\mathfrak C}_S$ is equivalent to consider  the number of components of $G[V\setminus S]$.
It was seen that, in $C_r(G)$,  each vertice containing $S$ 
establish a component. 
These vertices form a block
(that is, a maximal clique of $C_r(G)$)
and, by Theorem \ref{theo:relation-ktree-block}, we know that $|Q| -1 =\mu(S)$.
So, the factor becomes $\left( \sum_{\mu(S)+1} 1\right )^ {\mu(S)-1} $.
\qed
\end{proof}

\begin{theorem}
The determination of the number of clique trees of a $k$-tree can be performed in linear time complexity.
\end{theorem}
\begin{proof}
The determination of the minimal vertex separators and their multiplicities 
can be performed in $O(m)$ \cite{markenzon2010one}.
The number of maximal cliques of a $k$-tree $G$ with $n$ vertices is $n-k$.
Then any clique tree of $G$ has $n-k-1$ edges, 
each one of them corresponding to a minimal vertex separator of $\mathbb M$.
So, $\sum_{S \in \mathbb S} \mu(S) = n-k-1$.

Each factor of equation (\ref{eq:nossaformulacliquetree}) is $ \left(\mu(S) +1 \right)^{\mu(S)-1} $.
So, it corresponds to a product of $\mu(S_i)-1$ times the value $\mu(S_i) +1$;
each factor has $\mu(S_i)- 1$ factors.
Hence, the total number of factors is less than $\sum_{S\in{\mathbb S}} \mu(S)$, that, as it was seen,
is less than $n$.
Hence, the computation of equation (\ref{eq:nossaformulacliquetree}) has linear time complexity.
\hfill$\Box$
\end{proof}

\section{Number of  spanning trees of a block graph}

A \emph{spanning tree} of a graph $G$ is a tree containing all vertices of $G$. 
The number of spanning trees of $G$, $\tau(G)$, is a well known invariant of the literature, 
provided by the Kirchhoff's matrix tree theorem \cite{biggs1993algebraic}.
This is a very interesting result, since it combines structural and spectral aspects of the graph.
In this section, we present a new computation for the number of spanning trees of block graphs, 
showing that our result has a better complexity time than the general case.

\begin{theorem}\label{theo:equality}
If $G$ is a $k$-tree, then the number of clique trees of $G$ equals the number of spanning trees of $C_r(G)$.
\end{theorem}

\begin{proof}
Let $G$ be a $k$-tree. 
By Theorem \ref{theo:cardinalidadeSeparadoresKarvores}, all minimal vertex separators of a $k$-tree have cardinality $k$, 
so all edges of $C_r(G)$ have the same weight, $k$. 
Thus, all spanning trees of $C_r(G)$ have the same weight and therefore they are all maximum-weight spanning trees. 
So by Theorem \ref{theo:cliqueTree_spanninTree}, a tree $T$ is a clique tree of $G$ if and only if $T$ is a spanning tree of $C_r(G)$. 
Thus, counting the number of clique trees of $G$ is the same as counting the number of spanning trees of $C_r(G)$, which in turn, by Theorem \ref{theo:carac_cr_klinha}, is the same as counting the number of spanning trees of $\ell_{k+1}(G)$.
\qed
\end{proof}

\bigskip

Actually, Theorem \ref{theo:equality} can be applied to determine the  number of spanning trees of any block graph.

\begin{theorem}\label{theo:block-graph-k-tree}
Every connected block graph is the reduced clique graph of some $k$-tree.
\end{theorem}

\begin{proof}
Let $H$ be a connected block graph with $n$ vertices. 
So, it exists an integer $p\geq 0$ such that $H$ is $K_{1,p+2}$-free. If $p\geq 1$, then consider $k=p$. 
By Theorem \ref{teo:grafo(k+1)linhaDeKarvore}, $H$ is the $(k+1)$-line graph of a $k$-tree, say $G$, whence, by Theorem \ref{theo:carac_cr_klinha}, $H$ is also the reduced clique graph of $G$. 
If $p=0$, then $H$ is a complete graph $K_n$ and, for all $k\geq2$, it is the $(k+1)$-line graph 
of the $k$-star (a $k$-tree with $|{\mathbb S}| \leq 1$) with $n+k$ vertices.  \qed
\end{proof}

The \textit{Laplacian matrix} of the graph $G$ is the matrix $\mathbf{L}(G)=\mathbf{D}(G)-\mathbf{A}(G)$, where $\mathbf{A}(G)$ and $\mathbf{D}(G)$ are the adjacency matrix and the diagonal matrix of the vertex degrees of $G$, respectively. The Laplacian eigenvalues of $G$, $\mu_1 \geq \mu_2 \geq ...\geq \mu_{n}$, are the eigenvalues of $\mathbf{L}(G)$. It is known that the matrix $\mathbf{L}(G)$ is positive semidefinite and symmetric. Also, $\mu_n=0$  and the multiplicity of eigenvalue zero is equal to the number of connected components of $G$.
As a consequence of Kirchhoff's matrix tree theorem, the number of spanning trees of $G$ can be expressed by its Laplacian eigenvalues, as seen in Theorem \ref{kirch_auto}.

\begin{theorem}{\rm \cite{biggs1993algebraic}}\label{kirch_auto}
Let $G$ be a connected graph. If $\mu_1, \mu_2, ...,\mu_{n-1}$ are the nonzero eigenvalues of $\mathbf{L}(G)$, then
$$\tau(G)=\frac{1}{n} \prod_{i=1}^{n-1} \mu_i .$$
\end{theorem}

Theorem \ref{theo:nro-spanning-trees} is an immediate consequence of Theorems 
\ref{theo:relation-ktree-block}, \ref{theo:nro-clique-tree}, 
 \ref{theo:equality}  and \ref{theo:block-graph-k-tree}.

\begin{theorem}\label{theo:nro-spanning-trees}
Let $G$ be a connected block graph and $\mathbb Q$ its set of maximal cliques.
The number of spanning trees of  $G$ is
\begin{equation}\label{eq:formulaspanningtreeblockgraph}
\tau(G) = \prod_{Q \in {\mathbb Q}} \left(|Q| \right)^{|Q|-2}.
\end{equation}
\end{theorem}

The conclusion of Theorem \ref{theo:nro-spanning-trees} generalizes the expression obtained in 
Theorem 4  of \cite{abiad2017Wiener} (see formula (3)),  which furnishes  the number of spanning trees of a connected block graph with all blocks of  same size.

\begin{corollary}
Let $G$ be a connected block graph and $\mathbb Q$ its set of maximal cliques. If $\mu_1, \mu_2, ...,\mu_{n-1}$ are the nonzero eigenvalues of $\mathbf{L}(G)$, then
$$\frac{1}{n} \prod_{i=1}^{n-1} \mu_i = \prod_{Q \in {\mathbb Q}} \left(|Q| \right)^{|Q|-2} .$$
\end{corollary}

The efficiency of the determination of the number of  spanning cliques of a connected block graph 
is proved in  Theorem \ref{theo:number-spanning-trees}.
It relies on a compact representation of chordal graphs,
presented in  Markenzon {\em et al.} \cite{markenzon2013efficient}.
Based on this representation,  it was possible to state the next theorem.

\begin{theorem}{\rm \cite{markenzon2013efficient}}\label{theo:sum-max-cliques}
Let $G=(V,E)$ be a connected chordal graph and $\mathbb Q$ its set of maximal cliques.
 Then $$ \sum_{Q\in {\mathbb Q}} |Q| < n + m.$$
\end{theorem}

\begin{theorem}\label{theo:number-spanning-trees}
The determination of the number of spanning trees of a connected block graph can be performed in linear time complexity.
\end{theorem}
\begin{proof}
The determination of the maximal cliques  
is performed in $O(m)$ complexity time \cite{markenzon2010one}.
Let  ${\mathbb Q} = Q_1, \ldots Q_{|{\mathbb Q|}}$. 
Each factor of equation (\ref{eq:formulaspanningtreeblockgraph}) is $(|Q_i|)^{|Q_i|-2}$.
So, it corresponds to a product of $|Q_i|-2$ times a constant value, $|Q_i|$;
each factor has $|Q_i|-2$ factors.
Hence, the total number of factors is less than $\sum_{Q\in{\mathbb Q}} |Q|$, that, by Theorem \ref{theo:sum-max-cliques},
is less than $n+m$.
\qed
\end{proof}

\begin{corollary}
Let $G$ be a connected block graph. 
If $\mu_1, \mu_2, ...,\mu_{n-1}$ are the nonzero eigenvalues of $\mathbf{L}(G)$ then 
$\prod_{i=1}^{n-1} \mu_i $ can be determined in linear time complexity.
\end{corollary}

\section*{Acknowledgments}

This work was supported by grant 304706/2017-5, CNPq, Brazil.

\bibliographystyle{unsrtnat}
\bibliography{Arx-MarOliVin-2023}  






\end{document}